\documentclass[12pt,a4paper]{amsart}
\usepackage{amsmath,amsthm,amssymb,amsfonts}

\usepackage[pdftex]{color}
\usepackage[bookmarks=true,hyperindex,pdftex,colorlinks,citecolor=blue,linkcolor=blue,urlcolor=blue]
{hyperref}

\theoremstyle{plain}
\newtheorem{theorem}{Theorem}[section]
\newtheorem{cor}[theorem]{Corollary}
\newtheorem{prop}[theorem]{Proposition}
\newtheorem{lemma}[theorem]{Lemma}

\theoremstyle{definition}

\newtheorem{mydef}[theorem]{Definition}

\newtheorem{example}[theorem]{Example}
\newtheorem{examples}[theorem]{Examples}

\newtheorem{remark}[theorem]{Remark}

\newcommand{\R}{\mathbb{R}}
\newcommand{\N}{\mathbb{N}}

\newcommand{\C}{\mathbb{C}}
\newcommand{\K}{\mathbb{K}}

\newcommand{\e}{\varepsilon}
\newcommand{\eps}{\varepsilon}

\DeclareMathOperator{\supp}{supp}

\DeclareMathOperator{\re}{Re}

\DeclareMathOperator{\NA}{NA}
\DeclareMathOperator{\Id}{Id}

\renewcommand{\leq}{\leqslant}
\renewcommand{\geq}{\geqslant}
\renewcommand{\le}{\leqslant}
\renewcommand{\geq}{\geqslant}

\title{The Bishop-Phelps-Bollob\'as property for compact operators}

\author[S.~Dantas]{Sheldon Dantas}
\address{Departamento de An\'{a}lisis Matem\'{a}tico,
Universidad de Valencia, Doctor Moliner 50, 46100 Burjasot
(Valencia), Spain} \email{sheldon.dantas@uv.es}

\author[D.~Garc\'ia]{Domingo Garc\'ia}
\address{Departamento de An\'{a}lisis Matem\'{a}tico,
Universidad de Valencia, Doctor Moliner 50, 46100 Burjasot
(Valencia), Spain} \email{domingo.garcia@uv.es}

\author[M.~Maestre]{Manuel Maestre}
\address{Departamento de An\'{a}lisis Matem\'{a}tico,
Universidad de Valencia, Doctor Moliner 50, 46100 Burjasot
(Valencia), Spain} \email{manuel.maestre@uv.es}

\author[M.~Mart\'in]{Miguel Mart\'{\i}n}
\address[Mart\'{\i}n]{Departamento de An\'{a}lisis Matem\'{a}tico, Facultad de
 Ciencias, Universidad de Granada, 18071 Granada, Spain.
\href{http://orcid.org/0000-0003-4502-798X}{ORCID: \texttt{0000-0003-4502-798X} }
 }
\email{mmartins@ugr.es}

\thanks{First author supported by MINECO
MTM2014-57838-C2-2-P and CAPES, Doutorado Pleno CSF, BEX 0050/13-0. Second and third
authors supported by MINECO MTM2014-57838-C2-2-P and Prometeo II/2013/013. Four author supported by Spanish MINECO and FEDER project MTM2015-65020-P, and by Junta de Andaluc\'{\i}a and FEDER grant FQM-185.}


\date{April 3, 2016}


\dedicatory{}

\begin{document}

\begin{abstract}
We study the Bishop-Phelps-Bollob\'{a}s property (BPBp for short) for compact operators. We present some abstract techniques which allows to carry the BPBp for compact operators from sequence spaces to function spaces. As main applications, we prove the following results. Let $X$, $Y$ be Banach spaces. If $(c_0,Y)$ has the BPBp for compact operators, then so do $(C_0(L),Y)$ for every locally compact Hausdorff topological space $L$ and $(X,Y)$ whenever $X^*$ is isometrically isomorphic to $\ell_1$.
If $X^*$ has the Radon-Nikod\'{y}m property and $(\ell_1(X),Y)$ has the BPBp for compact operators, then so does $(L_1(\mu,X),Y)$ for every positive measure $\mu$; as a consequence, $(L_1(\mu,X),Y)$ has the the BPBp for compact operators when $X$ and $Y$ are finite-dimensional or $Y$ is a Hilbert space and $X=c_0$ or $X=L_p(\nu)$ for any positive measure $\nu$ and $1< p< \infty$.
For $1\leq p <\infty$, if $(X,\ell_p(Y))$ has the BPBp for compact operators, then so does $(X,L_p(\mu,Y))$ for every positive measure $\mu$ such that $L_1(\mu)$ is infinite-dimensional. If $(X,Y)$ has the BPBp for compact operators, then so do $(X,L_\infty(\mu,Y))$ for every $\sigma$-finite positive measure $\mu$ and $(X,C(K,Y))$ for every compact Hausdorff topological space $K$.
 \end{abstract}

\maketitle

\section{Introduction}
The study of norm-attaining operators goes back to J.~Lindenstrauss, who in 1963 \cite{L} initiated the study of pairs of Banach spaces $X$ and $Y$ for which the set of norm-attaining operators from $X$ into $Y$ is dense, trying to extend to operators the classical Bishop-Phelps theorem about the density of norm-attaining functionals. For a Banach space $X$ over $\K$ ($\R$ or $\C$), we write $S_X$, $B_X$, $X^*$ to denote, respectively, the unit sphere, the closed unit ball, and the topological dual of $X$. If $Y$ is also a Banach space, $\mathcal{L}(X,Y)$ is the space of all bounded linear operators from $X$ into $Y$ and $\mathcal{K}(X,Y)$ is its subspace consisting of all compact linear operators (recall that a linear operator $T:X\longrightarrow Y$ is said to be compact if $T(B_X)$ is relatively compact in $Y$). For $T\in \mathcal{L}(X,Y)$, $T^*$ will denote the adjoint of $T$. An operator $T\in \mathcal{L}(X,Y)$ is said to attain its norm if there is $x\in S_X$ such that $\|T\|=\|Tx\|$; we write $\NA(X,Y)$ to denote the set of all operators attaining their norms. With this notation, the Bishop-Phelps theorem \cite{BP} states that $\NA(X,\K)$ is dense in $X^*$ for every Banach space $X$. Lindenstrauss \cite{L} showed that there are Banach spaces $X$ and $Y$ such that $\NA(X,Y)$ is not dense in $\mathcal{L}(X,Y)$, and also gave some particular cases in which such a density holds: if $X$ is reflexive or if $Y$ is a closed subspace of $\ell_\infty$ containing the canonical copy of $c_0$ (actually, if $Y$ has property $\beta$, see definition below). A detailed account of known results in this area can be found in \cite{Acosta-RACSAM}.

Looking for norm-attaining operators, it is easier if we deal with compact operators as, for instance, compact operators from reflexive spaces always attain their norms (actually, James' theorem assures that this fact characterizes reflexivity). Many results about density of norm-attaining compact operators were given in the 1970's. For example, $\NA(X,Y)\cap \mathcal{K}(X,Y)$ is dense in $\mathcal{K}(X,Y)$ whenever one of the spaces $X$, $X^*$, $Y$ or $Y^*$ is isometrically isomorphic to a $L_1(\mu)$-space \cite{Johnson}. It was actually conjectured that compact operators between Banach spaces can be always approximated by norm-attaining (compact) operators, but it has been recently shown that this is not the case \cite{M}. We refer to the survey paper \cite{M1} for a detailed account on this subject.

B.~Bollob\'{a}s gave in 1970 \cite{Bol} a refinement of the Bishop-Phelps theorem in which both functionals and points where they almost attain the norm can be simultaneously approximated by norm-attaining functionals and points where they attain their norm.
In 2008, M.~D.~Acosta, R.~ M.~Aron, D.~Garc\'ia and M.~Maestre \cite{AAGM} introduced the Bishop-Phelps-Bollob\'{a}s property to study the operator version of Bollob\'{a}s' result.

\begin{mydef}[Bishop-Phelps-Bollob\'as property \cite{AAGM}]\label{def:BPBp} A pair of Banach spaces $(X, Y)$ has the \emph{Bishop-Phelps-Bollob\'as property} (\emph{BPBp} for short) if given $\e > 0$, there exists $\eta(\e) > 0$ such that whenever $T \in \mathcal{L}(X, Y)$ with $\|T\|=1$ and $x_0 \in S_X$ satisfy
	\begin{equation*}
	\|T(x_0)\| > 1 - \eta(\e),
	\end{equation*}
there are $S \in \mathcal{L}(X, Y)$ and $x_1 \in S_X$ such that
\begin{equation*}
\|S\|=\|S(x_0)\| = 1, \ \ \|x_0 - x_1\| < \e \ \ \mbox{and} \ \ \|S - T\| < \e.
\end{equation*}
In this case, we say that the pair $(X,Y)$ has the BPBp \emph{with the function} $\eps\longmapsto \eta(\eps)$.
\end{mydef}

With this definition, the refinement of Bollob\'{a}s \cite{Bol} of the Bishop-Phelps theorem just says that the pair $(X,\K)$ has the BPBp for every Banach space $X$.

There has been an extensive research on this topic, see \cite{Acosta-BJMA2016,8authors,ABGM-unif-convex,ABGKM-JMAA2014,AGKM-JMAA2016, ACKLM,CKG,ChoChoi,ChoiKim,ChoiKimLeeMartin, KL,KimLee-JMAA2015,KimLeeMartin-JMAA2015}, among others, where we refer for more information and background. Let us comment that there are many cases in which the density of norm-attaining operators between two Banach spaces $X$ and $Y$ carries to the fact that the pair $(X,Y)$ has the BPBp. For instance, among classical spaces, we have that the pairs $(L_p(\mu),L_q(\nu))$ have the BPBp whenever $\mu$ and $\nu$ are $\sigma$-finite measures and $1\leq p<\infty$ and $1\leq q\leq \infty$ \cite{ABGM-unif-convex,ChoiKim,ChoiKimLeeMartin,KL}, or $p=\infty$ and $1<q<\infty$ (actually, $(C(K),L_q(\mu))$ has the BPBp for every compact topological Hausdorff space $K$ \cite{KimLee-JMAA2015}); in the real case, the pair $(C_0(L_1),C_0(L_2))$ has the BPBp for every locally compact Hausdorff topological spaces $L_1$ and $L_2$ \cite{8authors}; in the complex case, the pair $(C(K),L_1(\mu))$ has the BPBp for every compact Hausdorff topological space $K$ and every measure $\mu$ \cite{Acosta-BJMA2016}. For general Banach spaces, we would like to present one result for domain spaces and one for range spaces. We need a couple of definitions. Let $Z$ be a Banach space. We say that $Z$ is \emph{uniformly convex} if for every sequences $\{x_n\}$, $\{y_n\}$ of elements of $B_Z$ with $\|x_n+y_n\|\longrightarrow 2$, one has $\|x_n-y_n\|\longrightarrow 0$; this is the case of the $L_p(\mu)$ spaces for $1<p<\infty$. The space $Z$ has property $\beta$ with constant $0\leq \rho<1$ if there are two sets $\{ z_i\,:\, i \in I\} \subset S_Z$, $\{ z^{*}_i\,:\, i \in I \} \subset S_{Z^*}$ such that
\begin{enumerate}
\item[$(i)$] $z^{*}_i(z_i)=1$, $\forall i \in I $;
\item[$(ii)$] $|z^*_i(z_j)|\leq \rho <1$ if $i, j \in I, i \ne j$;
\item[$(iii)$] for every $z\in Z$, $\|z\|=\sup\nolimits_{i \in I}  \bigl|z^{*}_i(z)\bigr|$ or, equivalently, $B_{Z^*}$ is the absolutely weakly$^*$-closed convex hull of $\{z^*_i\,:\, i\in I\}$.
\end{enumerate}
Examples of Banach spaces with property $\beta$ are those subspaces of $\ell_\infty$ containing the canonical copy of $c_0$ and finite-dimensional polyhedral spaces.

The two promised results about the BPBp are the following. Let $X$, $Y$ be Banach spaces.
\begin{itemize}
\item If $X$ is uniformly convex, then $(X,Z)$ has the BPBp for every Banach space $Z$ (\cite[Theorem~3.1]{KL} or \cite[Corollary~2.3]{ABGM-unif-convex}).
\item If $Y$ has property $\beta$, then $(Z,Y)$ has the BPBp for every Banach space $Z$ (\cite[Theorem~2.2]{AAGM}).
\end{itemize}

On the other hand, it was known from the seminal paper \cite{AAGM} that it is not always true that the density of $\NA(X,Y)$ implies the BPBp for the pair $(X,Y)$. There are some examples showing that, but maybe one of the more remarkable ones is the following.

\begin{example}[\mbox{\cite[Example~4.1]{ACKLM}}] \label{example-NAdense-no-BPBp}
{\slshape There exists a sequence of two-dimensional polyhedral spaces such that, writing $\mathcal{Y}$ to denote its $c_0$-sum, the pair $(\ell_1^2,\mathcal{Y})$ fails the BPBp. It is remarkable to say that, on the one hand, $\NA(\ell_1^2,Y)=\mathcal{L}(\ell_1^2,Y)$ for every Banach space $Y$ and, on the other hand, $\NA(X,\mathcal{Y})$ is dense in $\mathcal{L}(X,\mathcal{Y})$ for every Banach space $X$.}
\end{example}

This shows clearly that the study of the BPBp is not merely a trivial extension of the corresponding study of the density of norm-attaining operators, as there are more geometrical consequences. As an example, let us mention that the fact that a pair of the form $(\ell_1,Y)$ has the BPBp has been characterized in terms of the geometry of the space $Y$ by the following property.

\begin{mydef}[\mbox{\cite{AAGM}}] \label{defAHSP}
A Banach space $Y$ has the \emph{approximate hyperplane series property} (\emph{AHSP}, in short) if for every $\varepsilon>0$ there exists $\eta>0$ such that given a sequence $(y_k) \subset S_Y$ and a  convex series $\sum_{k=1}^{\infty}\alpha_k$ such that
$$
\left\|\sum_{k=1}^{\infty}\alpha_k y_k\right\| > 1-\eta,
$$
there exist $A\subset \N$, $y^*\in S_{Y^*}$, and a subset $\{z_k\, :\, k\in A\}\subset S_Y$ satisfying that
$$
\sum_{k\in A} \alpha_k>1-\varepsilon,\quad \|z_k-y_k\|<\varepsilon \quad \text{and} \quad y^*(z_k)=1
$$
for every $k\in A$.
\end{mydef}

It is shown in \cite[Theorem~4.1]{AAGM} that a Banach space $Y$ has the AHSP if and only if the pair $(\ell_1,Y)$ has the BPBp. Examples of spaces with this property are $L_1(\mu)$-spaces, $C(K)$-spaces, finite-dimensional spaces, uniformly convex spaces, among others.

In this paper we study Bishop-Phelps-Bollob\'{a}s type theorems for compact operators. Let us introduce the following definition, which has already appeared (mostly without name) in some of the references cited before.

\begin{mydef}[Bishop-Phelps-Bollob\'as property for compact operators]\label{def-BPBp-K} We say that a pair of Banach spaces $(X, Y)$ has the \emph{Bishop-Phelps-Bollob\'as property for compact operators} (\emph{BPBp for compact operators}) if given $\e > 0$, there exists $\eta(\e) > 0$ such that whenever $T \in \mathcal{K}(X, Y)$ with $\|T\|=1$ and $x_0 \in S_X$ satisfy
	\begin{equation*}
	\|T(x_0)\| > 1 - \eta(\e),
	\end{equation*}
there are $S \in \mathcal{K}(X, Y)$ and $x_1 \in S_X$ such that
\begin{equation*}
\|S\|=\|S(x_0)\| = 1, \ \ \|x_0 - x_1\| < \e \ \ \mbox{and} \ \ \|S - T\| < \e.
\end{equation*}
In this case, we say that the pair $(X,Y)$ has the BPBp for compact operators \emph{with the function} $\eps\longmapsto \eta(\eps)$.
\end{mydef}

An extensive list of pairs of spaces having the BPBp for compact operators is given in the following list. Some of the results are stated as here in the corresponding reference; for other ones, the proof can be easily adapted to the compact operators case, taking into account that when one start with a compact operator there, all operators involved are also compact.

\begin{examples}\label{example-previouslyknown}
The pair $(X, Y)$ of Banach spaces has the BPBp for compact operators when
\begin{itemize}
\item[(a)] $X$ is arbitrary and $Y$ has property $\beta$  (adapting the proof of \cite[Theorem~2.2]{AAGM});
\item[(b)] $X$ is uniformly convex and $Y$ is arbitrary (using \cite[Corollary~2.3]{ABGM-unif-convex} or adapting the proof of \cite[Theorem~3.1]{KL});
\item[(c)] $X$ is arbitrary and $Y$ is a uniform algebra - in particular, $Y=C_0(L)$ for a locally compact Hausdorff topological space $L$ - \cite[R2 in page 380]{CKG};
\item[(d)] $X = L_1(\mu)$ and $Y = L_1(\nu)$ for arbitrary measures $\mu$ and $\nu$ (adapting the proof of \cite[Theorem~3.1]{ChoiKimLeeMartin});
\item[(e)] $X = L_1(\mu)$ and $Y = L_{\infty} (\nu)$ for any measure $\mu$ and any localizable measure $\nu$ (adapting the proof of \cite[Theorem~4.1]{ChoiKimLeeMartin});
\item[(f)] $X = C(K)$ and $Y = C(L)$ in the real case where $K$ and $L$ are compact Hausdorff topological spaces (adapting the proof of \cite[Theorem 2.5]{8authors});
\item[(g)] $X = C_0(L)$ and $Y$ is uniformly convex where $L$ is any locally compact Hausdorff topological space \cite[Theorem~3.3]{8authors};
\item[(h)] $X$ is arbitrary and $Y^*$ is isometrically isomorphic to a $L_1(\mu)$-space \cite[Theorem~4.2]{8authors}; in particular, if $Y=C_0(L)$ for a locally compact Hausdorff topological space $L$;
\item[(i)] $X=L_1(\mu)$ for an arbitrary measure and $Y$ having the AHSP \cite[Corollary 2.4]{ABGKM-JMAA2014}.
\end{itemize}
\end{examples}

Let us mention that {\slshape it is not true that the BPBp for compact operators implies the BPBp for operators.}\, Indeed, the pair $(L_1[0,1],C[0,1])$ has the BPBp for compact operators (by any of the assertions (c), (h) or (i) above). But, $\NA(L_1[0,1],C[0,1])$ is not dense in $\mathcal{L}(L_1[0,1],C[0,1])$ \cite{Schachermayer}, nor the more this pair has the BPBp. On the other hand, {\slshape we do not know whether the BPBp implies the BPBp for compact operators.}

Our aim in this paper is to present some techniques to produce pairs of Banach spaces having the BPBp for compact operators, and to give some applications of them. These techniques are based on two old results about norm-attaining compact operators by J.~Johnson and J.~Wolfe \cite{Johnson}. To give these techniques is the goal of section~\ref{sec:tools}. We next present in
section~\ref{sect:applications} some applications of the previous results which carry the BPBp for compact operators from some sequence spaces to function spaces.  Let us recall some useful notation. Let $Z$ be a Banach space, $m\in\N$ and $1\leq p \leq \infty$. By $\ell_p^m(Z)$ we denote the $\ell_p$-sum of $m$ copies of $Z$ and $\ell_p(Z)$ is the $\ell_p$-sum of a countable infinitely many copies of $Z$; $c_0(Z)$ denotes the $c_0$-sum of a countable infinitely many copies of $Z$. If $(\Omega,\Sigma,\mu)$ is a positive measure space, $L_p(\mu,Z)$ is the space of all strongly measurable functions $f:\Omega \longrightarrow A$ such that $\|f\|^p$ is integrable for $1\leq p <\infty$ or $f$ is essentially bounded for $p=\infty$, endowed with the natural corresponding $p$-norm.

The main applications given in section~\ref{sect:applications} are the following. Let $X$, $Y$ be Banach spaces.
\begin{itemize}
\item If $(c_0,Y)$ has the BPBp for compact operators, then so does $(C_0(L),Y)$ for every locally compact Hausdorff topological space $L$.
\item If $X^*$ is isometrically isomorphic to $\ell_1$ and $(c_0,Y)$ has the BPBp for compact operators, then so does $(X,Y)$; in particular, if $Y$ is uniformly convex, then $(X,Y)$ has the BPBp for compact operators.
\item If $(\ell_1(X),Y)$ has the BPBp for compact operators and $X^*$ has the Radon-Nikod\'{y}m property, then $(L_1(\mu,X),Y)$ has the BPBp for compact operators for every positive measure $\mu$; as a consequence, $(L_1(\mu,X),Y)$ has the BPBp for compact operators when:
    \begin{itemize}
    \item $X$ and $Y$ are finite-dimensional,
    \item $Y$ is a Hilbert space and $X=c_0$ or $X=L_p(\nu)$ for any positive measure $\nu$ and $1< p < \infty$.
    \end{itemize}
\item For $1\leq p <\infty$, if $(X,\ell_p(Y))$ has the BPBp for compact operators, then so does $(X,L_p(\mu,Y))$ for every positive measure $\mu$ such that $L_1(\mu)$ is infinite-dimensional.
\item If $(X,Y)$ has the BPBp for compact operators, then so does $(X,L_\infty(\mu,Y))$ for every $\sigma$-finite positive measure $\mu$.
\item If $(X,Y)$ has the BPBp for compact operators, then so does $(X,C(K,Y))$ for every compact Hausdorff topological space $K$.
\end{itemize}

To finish this introduction, we would like to mention that a routinely change of parameters in Definitions \ref{def:BPBp} and \ref{def-BPBp-K} allows us to show that we may require the conditions not only for norm-one operators and vectors, but for operators and vectors with norm less than or equal to one. Here is the concrete statement of this result, which we will use without explicit mention.

\begin{remark}\label{BPBpLemma} Let $X$, $Y$ be Banach spaces.
\begin{enumerate}
\item[(a)] $(X,Y)$ has the BPBp if given $\e > 0$, there exists $\eta(\e) > 0$ such that whenever $T \in \mathcal{L}(X, Y)$ with $\|T\|\leq 1$ and $x_0 \in B_X$ satisfy
	\begin{equation*}
	\|T(x_0)\| > 1 - \eta(\e),
	\end{equation*}
there are $S \in \mathcal{L}(X, Y)$ and $x_1 \in S_X$ such that
\begin{equation*}
\|S\|=\|S(x_0)\| = 1, \ \ \|x_0 - x_1\| < \e \ \ \mbox{and} \ \ \|S - T\| < \e.
\end{equation*}
\item[(b)] $(X,Y)$ has the BPBp for compact operators if given $\e > 0$, there exists $\eta(\e) > 0$ such that whenever $T \in \mathcal{K}(X, Y)$ with $\|T\|\leq 1$ and $x_0 \in B_X$ satisfy
	\begin{equation*}
	\|T(x_0)\| > 1 - \eta(\e),
	\end{equation*}
there are $S \in \mathcal{K}(X, Y)$ and $x_1 \in S_X$ such that
\begin{equation*}
\|S\|=\|S(x_0)\| = 1, \ \ \|x_0 - x_1\| < \e \ \ \mbox{and} \ \ \|S - T\| < \e.
\end{equation*}
\end{enumerate}
\end{remark}

\section{The tools}\label{sec:tools}
In this section we present some abstract results which will allow to transfer the BPBp for compact operators from sequence spaces to function spaces. We first deal with domain spaces, for which the results are based in \cite[Lemma~3.1]{Johnson}: if a Banach space $X$ admits a net of norm-one projections with finite rank whose adjoints pointwise converge to the identity in norm, then $\NA(X,Y)\cap \mathcal{K}(X,Y)$ is dense in $\mathcal{K}(X,Y)$ for every Banach space $Y$. For the BPBp this result is not valid, as the finite-dimensionality of the domain space does not guarantee the BPBp (see Example~\ref{example-NAdense-no-BPBp}), and we have to impose additional conditions.

The most general result that we have is the following, from which we will deduce some particular cases.

\begin{lemma}[Main technical lemma]\label{Main-prop} Let $X$ and $Y$ be Banach spaces. Suppose that there exists a function $\eta:\R^+\longrightarrow \R^+$ such that given $\delta\in \R^+$, $x_1^*,\ldots,x_n^*\in B_{X^*}$ and $x_0\in S_X$, we may find a norm-one operator $P\in \mathcal{L}(X,X)$ and a norm-one operator $i\in\mathcal{L}(P(X),X)$ such that
\begin{enumerate}
\item[(1)] $\|P^*x_j^* - x_j^*\|<\delta$ for $j=1,\ldots,n$;
\item[(2)] $\|i(Px_0) - x_0\|<\delta$;
\item[(3)] $P\circ i=\Id_{P(X)}$;
\item[(4)] the pair $(P(X),Y)$ has the BPBp for compact operators with the function $\eta$.
\end{enumerate}
Then, the pair $(X, Y)$ has the BPBp for compact operators.
\end{lemma}

\begin{proof} Let $\e > 0$ be given. Define
$$
\eta'(\eps)=\min\left\{\frac14 \eta\bigl(\eps/2\bigr),\,\eps/6 \right\}.
$$
Fix $T \in \mathcal{K}(X, Y)$ with $\|T\|=1$ and $x_0\in S_X$ such that $\|T x_0\|> 1 - \eta'(\eps)$. As $T^*(B_{Y^*})$ is compact, we may find $x_1^*,\ldots,x_n^*\in B_{X^*}$ such that
$$
\min_j \|T^*y^* - x_j^*\|<\eta'(\eps) \qquad (y^*\in B_{Y^*}).
$$
Let $P\in \mathcal{L}(X,X)$ and $i\in \mathcal{L}(P(X),X)$ satisfying (1)--(4) for $\delta=\eta'(\eps)$. Then, for every $y^*\in B_{Y^*}$, we have
\begin{align*}
\|T^* y^* -P^*T^*y^*\| & \leq \min_{j} \bigl( \|T^*y^* - x_j^*\| + \|x_j^* - P^*x_j^*\| + \|P^*x_j^* - P^*T^*y^*\|\bigr) \\ & < 3 \eta'(\eps).
\end{align*}
Therefore,
$$
\|T - TP\|= \|T^*-P^*T^*\|\leq  3 \eta'(\eps).
$$
Next, consider $\widetilde{T}=T|_{P(X)}\in \mathcal{K}(P(X),Y)$. Then, $\|\widetilde{T}\|\leq 1$ and
\begin{align*}
\|\widetilde{T}(Px_0)\| &\geq \|T x_0\| - \|T x_0 - TPx_0\| \\ & \geq \|T x_0\| - \|T-TP\|>1 - \eta'(\eps) - 3 \eta'(\eps)\geq 1 - \eta\bigl(\eps/2\bigr).
\end{align*}
As the pair $(P(X),Y)$ has the BPBp for compact operators with the function $\eta$, there exist $\widetilde{S}\in \mathcal{K}(P(X),Y)$ and $\widetilde{x_1}\in S_{P(X)}$ such that
$$
\|\widetilde{S}\|=1=\|\widetilde{S}(\widetilde{x_1})\|,\quad \|Px_0 - \widetilde{x_1}\|<\eps/2,\quad \|\widetilde{S} - \widetilde{T}\|<\eps/2.
$$
Finally, consider $S=\widetilde{S} P\in \mathcal{K}(X,Y)$ which satisfies $\|S\|\leq 1$ and consider $x_1=i(\widetilde{x_1})\in B_X$. First,
$$
\|S x_1\|=\|[\widetilde{S} P  i](\widetilde{x_1})\|=\|\widetilde{S}(\widetilde{x_1})\|=1,
$$
so $\|S\|=1=\|S x_1\|$ (in particular, $\|x_1\|=1$). Next,
\begin{align*}
\|x_1 - x_0\| &\leq \|i(\widetilde{x_1}) - i(P x_0)\| + \|i(P x_0) - x_0\| \\ &< \|\widetilde{x_1} - P x_0\| + \eta'(\eps)  < \eps/2 + \eps/6 < \eps.
\end{align*}
On the other hand,
\begin{align*}
\|S-T\| &\leq \|S - TP\| + \|TP - T\| \leq \|\widetilde{S}P - TP\| + 3\eta'(\eps) \\ &\leq \|\widetilde{S}-\widetilde{T}\| + 3\eta'(\eps) < \eps/2 + \eps/2 = \eps.\qedhere
\end{align*}
\end{proof}

A useful particular case of the above result is the following proposition.

\begin{prop}\label{Prop-family-projections} Let $X$ be Banach space for which there exists a net $\{P_{\alpha}\}_{\alpha \in \Lambda}$ of rank-one projections on $X$ such that $\{P_{\alpha}x\}  \longrightarrow x$ for all $x \in X$ and $\{P_{\alpha}^* x^*\} \longrightarrow x^*$ for all $x^* \in X^*$ in norm. If for a Banach space $Y$ there exists a function $\eta:\R^+\longrightarrow \R^+$ such that all the pairs $(P_{\alpha}(X), Y)$ with $\alpha\in \Lambda$ have the BPBp for compact operators with the function $\eta$, then the pair $(X,Y)$ has the BPBp for compact operators.
\end{prop}

The proof is just an application of Lemma~\ref{Main-prop}, where the operator $i:P(X)\longrightarrow X$ is the  inclusion.

The requirement for the space $X$ in the above proposition are fulfilled if $X$ has a shrinking monotone Schauder basis (i.e.\ a monotone Schauder basis such that the biorthogonal functionals are a basis of the dual of the space).

\begin{cor}\label{corollary-monotoneshrinkingbasis}
Let $X$ be a Banach space with a shrinking monotone Schauder basis and let $\{P_n\}_{n\in \N}$ be the sequence of natural projections associated to the basis. If for a Banach space $Y$ there exists a function $\eta:\R^+\longrightarrow \R^+$ such that all the pairs $(P_n(X), Y)$ with $n\in \N$ have the BPBp (for compact operators) with the function $\eta$, then the pair $(X,Y)$ has the BPBp for compact operators.
\end{cor}

Another particular case of Proposition~\ref{Prop-family-projections} is given by the following corollary.

\begin{cor} \label{corollary-ordered-projections}
Let $X$ be a Banach space. Let $\{P_{\alpha}\}_{\alpha\in \Lambda}$ be a net of norm-one projections on $X$ such that $\alpha \preceq \beta$ implies that $P_{\alpha} (X) \subset P_{\beta} (X)$ and that $\{P_{\alpha}^* x^*\} \longrightarrow x^*$ in norm for all $x^* \in X^*$. If for a Banach space $Y$ there exists a function $\eta:\R^+\longrightarrow \R^+$ such that all the pairs $(P_{\alpha}(X), Y)$ with $\alpha\in \Lambda$ have the BPBp for compact operators with the function $\eta$, then the pair $(X,Y)$ has the BPBp for compact operators.
\end{cor}

\begin{proof} We have to prove that $\{P_{\alpha} x\} \longrightarrow x$ in norm for all $x \in X$ and then the result is just an application of Proposition~\ref{Prop-family-projections}. This is surely well-known, but we have not found a concrete reference, so we include an easy argument for the sake of completeness. First, let us prove that $Z= \overline{\bigcup_{\alpha \in \Lambda} P_{\alpha}(X)}$ is the whole of $X$. Otherwise, there is a non-null element $x_0^*\in X^*$ which is zero on $Z$. By hypothesis, $\{P_{\alpha}^* x_0^*\} \longrightarrow x_0^*$, but $P_\alpha^* x_0^*=x_0^*\circ P_\alpha=0$, a contradiction. Now, it is routine, using that the images of the family of projections is increasing, to prove that it converges pointwise in norm to the identity, as needed.
\end{proof}

Our next abstract result deals with range spaces instead of domain spaces. The idea of the proof, which is an adaptation to the BPBp of \cite[Lemma~3.4]{Johnson}, was used in \cite[Theorem~4.2]{8authors} to prove that every pair $(X,Y)$ has the BPBp for compact operators when $Y^*$ is isometric to an $L_1(\mu)$ space.

\begin{prop}\label{prop-range-space} Let $X$ and $Y$ be Banach spaces. Suppose that there exists a net of norm-one projections $\{Q_{\lambda}\}_{\lambda\in \Lambda}\subset \mathcal{L}(X,Y)$ such that $\{Q_{\lambda} y\} \longrightarrow y$ in norm for every $y\in Y$. If there is a function $\eta:\R^+\longrightarrow \R^+$ such that the pairs $(X, Q_{\lambda} (Y))$ with $\lambda\in \Lambda$ have the BPBp for compact operators with the function $\eta$, then the pair $(X, Y)$ has the BPBp for compact operators.
\end{prop}

\begin{proof}
Write $\eta'(\eps)=\frac12 \min\left\{\eta(\eps/2),\,\eps\right\}$. Let $T\in \mathcal{K}(X,Y)$ with $\|T\|=1$ and $x_0\in S_X$ such that
$$
\|T x_0\|>1-\eta'(\eps).
$$
As $T(B_X)$ is relatively compact, we may find $y_1,\ldots,y_m\in Y$ such that
$$
\min_{j} \|Tx - y_j\|<\eta'(\eps)/3
$$
for every $x\in B_X$. By hypothesis, there is $\lambda\in \Lambda$ such that
$$
\|Q_{\lambda}(y_j)-y_j\|<\eta'(\eps)/3 \qquad (j=1,\dots,m).
$$
Now, for every $x\in B_X$, we have
\begin{align*}
\|Tx-Q_{\lambda}Tx\|&\leq \min_{j} \|Tx - y_j\| + \|y_j-Q_\lambda(y_j)\| + \|Q_{\lambda}(y_j)-Q_\lambda T x\| \\ & < \min_j 2\|Tx-y_j\| + \eta'(\eps)/3 < \eta'(\eps).
\end{align*}
Therefore,
$$
\|Q_\lambda T - T\| \leq \eta'(\eps).
$$
Then we have that $\widetilde{T}=Q_\lambda T\in \mathcal{K}(X,Q_\lambda(Y))$ with $\|\widetilde{T}\|\leq 1$ and
$$
\|\widetilde{T}(x_0)\| \geq \|Tx_0\| - \|Q_\lambda T - T\|> 1 -2\eta'(\eps)\geq 1 - \eta(\eps/2).
$$
Then, there exists $\widetilde{S}\in \mathcal{K}(X,Q_\lambda(Y))$ with $\|\widetilde{S}\|=1$ and $x_1\in S_X$ such that
$$
\|\widetilde{S}(x_1)\|=1,\qquad \|\widetilde{S} - \widetilde{T}\|<\eps/2,\quad \text{and} \quad \|x_0-x_1\|<\eps/2<\eps.
$$
If we write $S\in \mathcal{K}(X,Y)$ to denote the operator $\widetilde{S}$ viewed as an operator with range in $Y$, we have that
\begin{equation*}
\|S\|=1=\|Sx_1\|\quad \text{and} \quad \|S-T\|\leq \|\widetilde{S} - \widetilde{T} \| + \|\widetilde{T} - T\|<\eps/2 + \eta'(\eps)\leq \eps.\qedhere
\end{equation*}
\end{proof}

We finish this section about technical results with an extension of some results from \cite[\S 2]{ACKLM} to compact operators.

\begin{lemma}\label{lemma-X-Y_1-p-Y_2}
Let $X$, $X_1$, $X_2$, $Y$, $Y_1$, $Y_2$ be Banach spaces.
\begin{enumerate}
\item[(a)] If one of the the pairs $(X_1\oplus_1 X_2,Y)$ or $(X_1\oplus_\infty X_2,Y)$ has the BPBp for compact operators with a function $\eta$, then so do the pairs $(X_j,Y)$ with $j=1,2$ with the same function $\eta$.
\item[(b)] If one of the pairs $(X,Y_1 \oplus_1 Y_2)$ $(X,Y_1 \oplus_\infty Y_2)$ has the BPBp for compact operators with a function $\eta$, then so do the pairs $(X,Y_j)$ with $j=1,2$ with the the same function $\eta$.
\end{enumerate}
\end{lemma}

\begin{proof} For (a) one just has to adapt to compact operators the proof of \cite[Proposition~2.3]{ACKLM} for $\oplus_1$ and  \cite[Proposition~2.6]{ACKLM} for $\oplus_\infty$.

(b). The result follows again adapting to compact operators the proofs of \cite[Proposition~2.7]{ACKLM} for $\oplus_1$ and \cite[Proposition~2.3]{ACKLM} for $\oplus_\infty$.
\end{proof}

\section{Applications}\label{sect:applications}
The first main application of the results in the previous section is the following sufficient condition for a pair of the form $(C_0(L),Y)$ to have the BPBp for compact operators.

\begin{theorem}\label{theorem-C0Y}
Let $L$ be a locally compact Hausdorff topological space and let $Y$ be a Banach space. If $(c_0,Y)$ has the BPBp for compact operators, then $(C_0(L),Y)$ has the BPBp for compact operators.
\end{theorem}

The result will be proved by applying Lemma~\ref{Main-prop} and, to do so, we need two preliminary results. The first one is the following lemma, which we only need for $X=\K$, but we state in the general form for completeness.

\begin{lemma}\label{lemma-c0Y}
Let $X$, $Y$ be Banach spaces. Then the following are equivalent:
\begin{enumerate}
\item[(i)] The pair $(c_0(X),Y)$ has the BPBp for compact operators;
\item[(ii)] there is a function $\eta:\R^+\longrightarrow \R^+$ such that the pairs $(\ell_\infty^{m}(X),Y)$ with $m\in \N$ have the BPBp for compact operators with the function $\eta$.
\end{enumerate}
Moreover, when $\mathcal{K}(X,Y)=\mathcal{L}(X,Y)$ (in particular, if one of the spaces $X$ or $Y$ is finite-dimensional), this happens when $(c_0(X),Y)$ or $(\ell_\infty(X),Y)$ has the BPBp.
\end{lemma}

\begin{proof}
(i) implies (ii) follows from Lemma~\ref{lemma-X-Y_1-p-Y_2}.(a), as each $\ell_\infty^{m}(X)$ is an $\ell_\infty$-summand in $c_0(X)$. (ii) implies (i) is an easy consequence of Corollary~\ref{Prop-family-projections}.

Finally, when $\mathcal{K}(X,Y)=\mathcal{L}(X,Y)$, if $(c_0(X),Y)$ or $(\ell_\infty(X),Y)$ has the BPBp, then (ii) holds by using \cite[Proposition~2.6]{ACKLM} and the fact that every operator from $\ell_\infty^m(X)$ into $Y$ is compact.
\end{proof}

It follows, in particular, the following consequence.

\begin{cor}
Let $Y$ be a Banach space. If the pair $(c_0,Y)$ has the BPBp, then it has the BPBp for compact operators.
\end{cor}

The next preliminary result is based on \cite[Proposition~3.2]{8authors} and gives the possibility of applying Lemma~\ref{Main-prop} when the domain space is a $C_0(L)$ space.

\begin{lemma}[Extension of \mbox{\cite[Proposition~3.2]{8authors}}]\label{lemma:C0LXY-projections}
Let $L$ be a locally compact Hausdorff topological space. Given $\delta>0$, $\mu_1,\ldots,\mu_n\in B_{C_0(L)^*}$ and $f_0\in B_{C_0(L)}$, there exist a norm-one projection $P\in \mathcal{L}(C_0(L),C_0(L))$ and a norm-one operator $i\in \mathcal{L}(P(C_0(L)),C_0(L))$ such that:
\begin{itemize}
\item[(1)] $\|P^*\mu_j - \mu_j\|<\delta$ for $j=1,\ldots,n$;
\item[(2)] $\|i(P f_0)-f_0\|<\delta$;
\item[(3)] $P\circ i=\Id_{P(C_0(L))}$;
\item[(4)] $P(C_0(L))$ is isometrically isomorphic to  $\ell_\infty^{m}$ for some $m\in \N$.
\end{itemize}
\end{lemma}

\begin{proof}
Almost everything is given by \cite[Proposition~3.2]{8authors} (and its proof), and only the operator $i$ has to be defined, but we need to give the details to do so. First, we may and do suppose that $\|f_0\|=1$ (indeed, if $f_0=0$, then (2) is always true and, as $P$ is a projection, (3) is true by just taking $i$ to be the inclusion of $P(C_0(L))$ into $C_0(L)$; otherwise, use $f_0/\|f_0\|$ and the result for $f_0$ will follows). Take $\mu_0\in S_{C_0(L)^*}$ such that $\mu_0(f_0)=\|f_0\|=1$. By the Riesz representation theorem, we may view $\mu_0,\mu_1,\ldots,\mu_n$ as Borel measures on $L$. Consider the finite positive regular measure $\mu=\sum_{j=0}^n |\mu_j|$. By using the Radon-Nikod\'{y}m theorem, the density of simple functions on $L_1(\mu)$, the regularity of $\mu$, Urysohn's lemma, and the continuity of $f_0$ (see the proof of \cite[Proposition~3.2]{8authors} for the details), we may find a finite collection $K_1,\ldots,K_m$ of pairwise disjoint compact subsets of $L$ with $\mu(K_k)>0$ for $k=1,\ldots, m$, a collection of continuous functions with pairwise disjoint compact support $\varphi_1,\ldots,\varphi_m$ with values in $[0,1]$ and such that $\varphi_k= 1$ on $K_k$ for every $k=1,\ldots,m$, in such a way that, if we define
$$
P(f)=\sum_{k=1}^m \frac{1}{\mu(K_k)}\left( \int_{K_k} f \,d\mu\right)\varphi_k \qquad \bigl(f\in C_0(L)\bigr),
$$
one has
\begin{enumerate}
\item[(a)] $P\in \mathcal{L}(C_0(L),C_0(L))$ is a norm-one projection;
\item[(b)] $\|P^*\mu_j - \mu_j\|<\delta/2$ for every $j=0,1,\ldots,n$;
\item[(c)] $P(C_0(L))$ is the linear span of $\{\varphi_1,\ldots,\varphi_m\}$ and so, it is isometrically isomorphic to $\ell_\infty^{m}$;
\item[(d)] $\sup\limits_{t,s\in K_k}|f_0(t)-f_0(s)|<\delta/2$ for $k=1,\ldots,m$;
\item[(e)] $\sup\left\{ \bigl|[Pf_0](t)-f_0(t)\bigr|\,:\,t\in\bigcup_{k=1}^m K_k\right\}<\delta/2$.
\end{enumerate}
Then, we have (1) and (4) of the lemma. Next, we use (b) with $j=0$ to get that
$$
\|P f_0\|\geq |\mu_0(P f_0)|\geq |\mu_0(f_0)|-\|P^*\mu_0 - \mu_0\|> 1 - \delta/2,
$$
and consider $\Upsilon_0\in P(X)^*$ such that $\|\Upsilon_0\|=1$ and $\Upsilon_0(P f_0)=\|P f_0\|>1-\delta/2$.  On the other hand, we may use (d) to get a compactly supported continuous function $\Psi:L\longrightarrow [0,1]$ such that $\Psi= 1$ on $\bigcup_{k=1}^m K_k$ and such that
$$
\sup_{t\in \supp(\Psi)} \bigl|[Pf_0](t)-f_0(t)\bigr| <\delta/2.
$$
We are now ready to define the operator $i\in \mathcal{L}(P(X),C_0(L))$ as follows:
$$
\left[i\left( \sum_{k=1}^m \alpha_k \varphi_k \right)\right](t)=\Psi(t)\left( \sum_{k=1}^m \alpha_k \varphi_k(t) \right)\, + \, \bigl(1-\Psi(t)\bigr)\Upsilon_0\left( \sum_{k=1}^m \alpha_k \varphi_k \right)\,f_0(t)
$$
for every $t\in L$ and every $\alpha_1,\ldots,\alpha_m\in \K$.
Then, $\|i\|\leq 1$, $P\circ i=\Id_{P(X)}$ (this gives (3)), and (2) is just the following computation:
\begin{equation*}
\|i(Pf_0)-f_0\|\leq \bigl\|\Psi [Pf_0 -f_0]\bigr\|\,   + \, \bigl\|(1-\Psi)\bigl[1-\|Pf_0\|]f_0\bigr\|<\delta/2 + \delta/2= \delta.\qedhere
\end{equation*}
\end{proof}

\begin{proof}[Proof of Theorem~\ref{theorem-C0Y}]
By Lemma~\ref{lemma:C0LXY-projections} and Lemma~\ref{lemma-c0Y}, we are in the hypotheses of Lemma~\ref{Main-prop}, so the result follows.
\end{proof}

A family of Banach spaces $Y$ for which $(c_0,Y)$ has the BPBp have been recently discovered \cite{AGKM-JMAA2016}, which strictly contains uniformly convex spaces and Banach spaces with property $\beta$. By using Theorem~\ref{theorem-C0Y}, one get that $(C_0(L),Y)$ have the BPBp for compact operators for all elements $Y$ of that family. We need a definition. Let $Y$ be a Banach space, $E \subset S_Y$ and $F: E \longrightarrow S_{Y^*}$. We say that the family $E$ is \emph{uniformly strongly exposed} by $F$ if for every
$\varepsilon > 0$ there is $\delta > 0$ such that
$$
y \in B_Y, \ \ e \in E, \ \ \re F(e)(y) > 1 - \delta \ \Rightarrow  \ \Vert y-e \Vert  < \varepsilon.
$$

The promised application is the following.

\begin{cor}\label{cor-beta-uniformlyconvex-C0LY}
Let $L$ be a locally compact Hausdorff topological space and let $Y$ be a Banach space. Suppose that there exist a set $I$, $\{ y_i : i \in I\}
\subset S_{Y}$, $\{ y_i^* : i \in I\} \subset S_{Y^*}$, a subset $E \subset S_Y$, a mapping $F: E \longrightarrow
S_{Y^*}$ and $0 \le \rho < 1$ satisfying that
\begin{enumerate}
\item[(1)] $y_{i}^* (y_i)=1, \forall i \in I$;
\item[(2)] $\vert y_{i}^* (y_j)\vert \le \rho, \forall i, j  \in I, i \ne j$;
\item[(3)] $E$ is uniformly strongly exposed by $F$;
\item[(4)]  $\vert F(e) (y_i)\vert \le \rho, \forall e \in E, i \in I$;
\item[(5)] for any $y\in Y$, $\Vert y \Vert = \max \bigl\{ \sup  \{ \vert y_{i}^* (y) \vert
: i \in I\}, \sup  \{ \vert F(e)  (y) \vert\, :\, e \in E\} \}$.
\end{enumerate}
Then, the pair $(C_0(L),Y)$ has the BPBp for compact operators.
\end{cor}

The proof is just an application of \cite[Theorem~2.4]{AGKM-JMAA2016} to get that $(c_0,Y)$ has the BPBp, Lemma~\ref{lemma-c0Y} and Theorem~\ref{theorem-C0Y}.

Observe that this result covers the already known cases of $Y$ being uniformly convex ($I=\emptyset$) and of $Y$ having property $\beta$ ($E=\emptyset$). It is proved in the cited paper \cite{AGKM-JMAA2016} that there are examples of Banach spaces $Y$ satisfying the requirements of Corollary~\ref{cor-beta-uniformlyconvex-C0LY} which are neither uniformly convex nor satisfy property $\beta$, even in dimension two.

Our next result about domain spaces deals with isometric preduals of $\ell_1$. We do not know whether it can be extended to general $L_1$-predual spaces.

\begin{theorem}\label{theorem-predualL1} Let $X$ be a Banach space such that $X^*$ is isometrically isomorphic to $\ell_1$ and let $Y$ be a Banach space. If the pair $(c_0, Y)$ has the BPBp for compact operators, then $(X, Y)$ has the BPBp for compact operators.
\end{theorem}

\begin{proof}
Let $\{e_n^*\}_{n\in \N}$ be a basis of $X^*$ isometrically equivalent to the usual $\ell_1$-basis, and let $Y_n$ be the linear span of $\{e_1^*,\ldots,e_n^*\}$ for every $n\in \N$. It is proved in \cite[Corollary~4.1]{Gasparis} that there exists a sequence of norm-one projections $Q_n:X^*\longrightarrow X^*$ such that $Q_n(X^*)=Y_n$ and $Q_nQ_{n+1}=Q_n$ for every $n\in \N$. It is also shown that writing $P_n$ for the restriction of $Q_n^*$ to $X$ and using the $w^*$-continuity of $Q_n$, one has that $P_n$ is a norm-one projection in $X$ whose range $E_n=Q_n^*(Y_n^*)\subset X$
is isometrically isomorphic to $\ell_\infty^n$; besides, as $Q_nQ_{n+1}=Q_n$, one has $P_{n+1}P_n=P_n$ and so, $P_{n}(X)\subseteq P_{n+1}(X)$.

Now, we may apply Lemma~\ref{lemma-c0Y} and Corollary~\ref{corollary-ordered-projections} to get that $(X,Y)$ has the BPBp for compact operators.
\end{proof}

As for pairs of the form $(C_0(L),Y)$ (see Corollary~\ref{cor-beta-uniformlyconvex-C0LY}), we have the following consequence.

\begin{cor}
Let $X$ be a Banach space such that $X^*$ is isometrically isomorphic to $\ell_1$ and let $Y$ be a Banach space. Suppose that there exist a set $I$, $\{ y_i : i \in I\}
\subset S_{Y}$, $\{ y_i^* : i \in I\} \subset S_{Y^*}$, a subset $E \subset S_Y$, a mapping $F: E \longrightarrow
S_{Y^*}$ and $0 \le \rho < 1$ satisfying that
\begin{enumerate}
\item[(1)] $y_{i}^* (y_i)=1, \forall i \in I$;
\item[(2)] $\vert y_{i}^* (y_j)\vert \le \rho, \forall i, j  \in I, i \ne j$;
\item[(3)] $E$ is uniformly strongly exposed by $F$;
\item[iv)]  $\vert F(e) (y_i)\vert \le \rho, \forall e \in E, i \in I$;
\item[(4)] for any $y\in Y$, $\Vert y \Vert = \max \bigl\{ \sup  \{ \vert y_{i}^* (y) \vert
: i \in I\}, \sup  \{ \vert F(e)  (y) \vert\, :\, e \in E\} \}$.
\end{enumerate}
Then, the pair $(X,Y)$ has the BPBp for compact operators.
\end{cor}

The proof is just an application of \cite[Theorem~2.4]{AGKM-JMAA2016} to get that $(c_0,Y)$ has the BPBp, Lemma~\ref{lemma-c0Y} and Theorem~\ref{theorem-predualL1}.

Observe again that this result covers the cases of $Y$ being uniformly convex ($I=\emptyset$) and of $Y$ having property $\beta$ ($E=\emptyset$). If $Y$ has property $\beta$, the result was already known (see Examples~\ref{example-previouslyknown}.(a)), but it was unknown for uniformly convex spaces.

\begin{cor}
Let $X$ be a Banach space such that $X^*$ is isometrically isomorphic to $\ell_1$ and let $Y$ be a uniformly convex Banach space. Then $(X, Y)$ has the BPBp for compact operators.
\end{cor}

Next we will give a result for $L_1(\mu,X)$-spaces as domain which is an extension of Examples~\ref{example-previouslyknown}.(i).

\begin{theorem}\label{thr-L1muX} Let $\mu$ be a positive measure, let $X$ be a Banach space such that $X^*$ has the Radon-Nikod\'{y}m property and let $Y$ be a Banach space. If $(\ell_1(X),Y)$ has the BPBp for compact operators, then the pair $(L_1(\mu,X),Y)$ has the BPBp for compact operators.
\end{theorem}

We first need the following lemma, which gives a version for compact operators of \cite[Theorem~6]{KimLeeMartin-JMAA2015}. Observe that in this case, no assumption on $X$ is needed.

\begin{lemma}\label{lemma-characterizationell1X-Y}
Let $X$, $Y$ be Banach spaces. Then the following are equivalent:
\begin{enumerate}
\item[(i)] for every $\varepsilon>0$ there exists $0<\xi(\varepsilon)<\varepsilon$ such that given sequences $(T_k)\subset B_{\mathcal{K}(X,Y)}$ and $(x_k) \subset B_X$, and a convex series $\sum_{k=1}^\infty\alpha_k$ such that
$$
\left\|\sum_{k=1}^{\infty}\alpha_k T_k x_k\right\| > 1-\xi(\varepsilon),
$$
there exist a subset $A\subset \N$ finite, $y^*\in S_{Y^*}$ and sequences $(S_k)\subset S_{\mathcal{K}(X,Y)}$, $(z_k) \subset S_X$ satisfying the following:
\begin{enumerate}
\item  $\sum_{k\in A} \alpha_k>1-\varepsilon$,
\item $ \|z_k-x_k\|<\varepsilon$ and $ \|S_k-T_k\|<\varepsilon$ for all $k\in A$,
\item $y^*(S_kz_k)=1$ for every $k\in A$.
\end{enumerate}
(in this case, we may say that the pair $(X,Y)$ has the \emph{generalized AHSP for compact operators});
\item[(ii)] the pair $(\ell_1(X),Y)$ has the BPBp for compact operators;
\item[(iii)] there is a function $\eta:\R^+\longrightarrow \R^+$ such that the pairs $(\ell_1^{m}(X),Y)$ with $m\in \N$ have the BPBp for compact operators with the function $\eta$.
\end{enumerate}
Moreover, if $\mathcal{K}(X,Y)=\mathcal{L}(X,Y)$ (in particular,
if one of the spaces $X$ or $Y$ is finite-dimensional), then the above is equivalent to
\begin{enumerate}
\item[(iv)] the pair $(\ell_1(X),Y)$ has the BPBp.
\end{enumerate}
\end{lemma}

\begin{proof}
(i) implies (ii). We can adapt the proof of Theorem~6 of \cite{KimLeeMartin-JMAA2015} to the case of compact operators since, as we suppose the set $A$ to be finite and all the operators $T_k$'s and $S_k$'s to be compact, the operator $S:\ell_1(X)\longrightarrow Y$ defined there is also compact. (ii) implies (iii) follows from Lemma~\ref{lemma-X-Y_1-p-Y_2}.(a), as each $\ell_1^{m}(X)$ is an $\ell_1$-summand in $\ell_1(X)$. Finally, for (iii) implies (i), we can again adapt the proof of \cite[Theorem~6]{KimLeeMartin-JMAA2015} to the case of compact operators, using that in our item (i) we may reduce to finite sums instead of series (using the analogous for compact operators of \cite[Remark~5.a]{KimLeeMartin-JMAA2015}).

If $\mathcal{K}(X,Y)=\mathcal{L}(X,Y)$, item (iii) is equivalent to the fact that all the pairs $(\ell_1^m(X),Y)$ with $m\in \N$ have the BPBp with the same function $\eta$. Then, it is shown in  \cite[Theorem~6]{KimLeeMartin-JMAA2015} that this is equivalent to (iv).
\end{proof}

In particular, we have the following characterization of when the pairs of the form $(\ell_1,Y)$ has the BPBp for compact operators.

\begin{cor}\label{cor:characAHSP}
Let $Y$ be a Banach space. Then, the following are equivalent:
\begin{enumerate}
\item[(i)] the pair $(\ell_1,Y)$ has the BPBp for compact operators;
\item[(ii)] $Y$ has the AHSP;
\item[(iii)] the pair $(\ell_1,Y)$ has the BPBp;
\item[(iv)] for every positive measure $\mu$, the pair $(L_1(\mu),Y)$ has the BPBp for compact operators;
\item[(v)] there is a positive measure $\mu$ such that $L_1(\mu)$ is infinite-dimensional and the pair $(L_1(\mu),Y)$ has the BPBp for compact operators.
\end{enumerate}
\end{cor}

\begin{proof}
(ii) is equivalent to (iii) by \cite[Theorem~4.1]{AAGM}. (i) is equivalent to (iii) by Lemma~\ref{lemma-characterizationell1X-Y}. (ii) implies (iv) is Examples~\ref{example-previouslyknown}.(i). (iv) implies (v) is obvious. Finally, (v) implies (ii) is proved in \cite[Corollary~2.4]{ABGKM-JMAA2014}.
\end{proof}

We will also need the following modification of \cite[Lemma~III.2.1, p.~67]{Diestel-Uhl-book}. For $1\leq p <\infty$, let $p^*$ be the conjugate exponent, i.e.\ $p^*=\infty$ for $p=1$ and $p^*$ is determined by the equation $1/p+1/p^*=1$ for $1<p<\infty$.

\begin{lemma}\label{lemma-P-P*-when-mu-finite}
Let $(\Omega, \Sigma, \mu)$ be a measure space such that $L_1(\mu)$ is infinite-dimensional, let $X$ be a Banach space and  $\e > 0$.
\begin{enumerate}
  \item[(a)]  For $1\leq p<\infty$, given $f_1,\ldots,  f_n$ in $L_p(\mu, X)$ there exists a norm-one projection $P:L_p(\mu, X)\longrightarrow L_p(\mu, X)$ such that $P(L_p(\mu, X))$ is isometrically isomorphic to $\ell_p(X)$ and
      $$\|P(f_j)-f_j\|<\e,$$ for every $j=1,\ldots, n$.
  \item[(b)] If $\mu$ is a finite measure then, for $1\leq p< \infty$,  $f_1,\ldots,  f_n$ in $L_p(\mu, X)$, $g_1,\ldots,  g_m$ in $L_{p^*}(\mu, X^*)$ and $\e>0$, there exists a norm-one projection $P:L_p(\mu, X)\longrightarrow L_p(\mu, X)$ such that $P(L_p(\mu, X))$ is isometrically isomorphic to $\ell_p(X)$ and
       such that
\begin{equation*}
 \|f_j - P f_j\|_p < \e	 \quad \ \mbox{and} \quad \  \| g_k - P^* g_k\|_{p^*} < \e,
	\end{equation*}
for all $j = 1, \ldots, n$ and $k = 1, \ldots, m$.	
  \item[(c)] If $\mu$ is a finite measure then, given $g_1,\ldots,  g_m$ in $L_\infty(\mu, X)$ and $\e>0$ there exists a norm-one projection $P:L_\infty(\mu, X)\longrightarrow L_\infty(\mu, X)$ such that $P(L_\infty(\mu, X))$ is isometrically isomorphic to $\ell_\infty(X)$ and
       such that
\begin{equation*}
  \| g_k - P g_k\|_{\infty} < \e,
	\end{equation*}
for all $k = 1, \ldots, m$.
\end{enumerate}
\end{lemma}

 \begin{proof}
Since $L_1(\mu)$ is infinite-dimensional, there exists a sequence $\mathcal{A}=(A_i)_{i=1}^\infty$ of pairwise disjoint measurable sets such that $0<\mu(A_i)<\infty$ for every $i\in \N$.

(a). Fix $1\leq p<\infty$. Let $f_1,\ldots,  f_n$ in $L_p(\mu, X)$. As, by definition, simple functions are dense in $L_p(\mu, X)$, for each $j\in\{1,\ldots,n\}$ we can find a finite family $\mathcal{A}_j$ of pairwise disjoint measurable sets of positive and finite measure, and $x_A\in X$ for every $A\in \mathcal{A}_j$ such that the vector-valued simple function     $\phi_j=\sum_{A \in \mathcal{A}_j}x_A\chi_A$ satisfies
$$\|f_j-\phi_j\|_p<\e$$
for every $j = 1, \ldots, n$.
Define $$\Omega_1=\left[\bigcup_{i=1}^\infty A_i\right]\cup\left[\bigcup_{j=1}^n\bigcup_{A\in \mathcal{A}_j}A\right]$$ and let consider $\mathcal{B}=\{B_i:\, i\in \N\}$ an infinite countable partition of $\Omega_1$ such that all their elements are measurable with $0<\mu(B_i)<\infty$ for every $i\in \N$ and such that $\mathcal{B}$ is a refinement of the above families. In the case $p=1$, clearly
$$
\sum_{i=1}^{\infty}\left\|\frac{1}{\mu(B_i)}\Big(\int_{B_i}f d\mu\Big)\, \chi_{B_i}\right\|_1= \sum_{i=1}^{\infty}\left\|\int_{B_i}f d\mu\right\|\leq \|f\|_1
$$
for every $f$ in $L_1(\mu, X)$. For $1<p<\infty$, let us observe that for any $r,s\in \N$ with $1\leq r< s$, the fact that $\mathcal{B}$ is a partition leads to
$$
\left\|\sum_{i=r}^{s}\frac{1}{\mu(B_i)}\Big(\int_{B_i}f d\mu\Big) \chi_{B_i}(t)\right\|^p= \sum_{i=r}^{s}\frac{1}{\mu(B_i)^p}\left\|\int_{B_i}f d\mu\right\|^p\chi_{B_i}(t)
$$
for every $t$ in  $\Omega$. Hence
\begin{align*}
\left\|\sum_{i=r}^{s}\frac{1}{\mu(B_i)}\Big(\int_{B_i}f d\mu\Big)\chi_{B_i}\right\|_p^p & =
\int_\Omega\Big( \sum_{i=r}^{s}\frac{1}{\mu(B_i)^p}\left\|\int_{B_i}f d\mu\right\|^p\chi_{B_i}\Big)d\mu
\\ &=\sum_{i=r}^{s}\frac{1}{\mu(B_i)^p}\left\|\int_{B_i}f d\mu\right\|^p\int_{\Omega}\chi_{B_i}d\mu.
\end{align*}
But
\begin{align*}
  \sum_{i=1}^{\infty}\frac{1}{\mu(B_i)^p}\left\|\int_{B_i}f d\mu\right\|^p\int_{\Omega}\chi_{B_i}d\mu
   & =\sum_{i=1}^{\infty}\mu(B_i)^{1-p}\left\|\int_{B_i}f d\mu\right\|^p \\
  &\leq  \sum_{i=1}^{\infty}\mu(B_i)^{1-p}\int_{B_i}\|f\|^p d\mu\left(\int_{\Omega}\chi_{B_i}d\mu\right)^{\frac{p}{p^*}}
   \\ & = \sum_{i=1}^{\infty}\int_{B_i}\|f\|^pd\mu\leq \|f\|_p^p
\end{align*}
for every $k$.

Thus, for $1\leq p<\infty$, we can define  $P:L_p(\mu, X)\longrightarrow L_p(\mu, X)$ by
$$P(f)=\sum_{i=1}^{\infty}\frac{1}{\mu(B_i)}\Big(\int_{B_i}f d\mu\Big)\,\chi_{B_i} \qquad \bigl(f\in L_p(\mu,X)\bigr).
$$
As, obviously, $P(\chi_{B_i})=\chi_{B_i}$ for every $i$, we have that  $P$ is a norm-one projection such that $P(L_p(\mu, X))$ is isometrically isomorphic to $\ell_p(X)$. Moreover, for each $j\in \{1,\ldots,n\}$, there exists a sequence $(x_{ij})_{i\in\N}$ in $X$ such that
    $$
    \phi_j=\sum_{i=1}^{\infty}x_{ij}\chi_{B_i},
    $$
where the equality holds both pointwise and with respect to the $p$-norm. This implies that
      $$\|P(f_j)-f_j\|_p=\|\phi_j-f_j\|_p<\e,$$ for every $j=1,\ldots, n$.

(b). If we now assume that $\mu$ is a finite measure, the sequence  $\mathcal{A}=(A_i)_{i=1}^\infty$ of pairwise disjoint sets of positive measure can be assumed to be a partition of $\Omega$. As above, given $f_1,\ldots,  f_n$ in $L_p(\mu, X)$, we can find $\phi_1,\ldots,\phi_n$ simple functions $\phi_j=\sum_{A \in \mathcal{A}_j}x_A\chi_A$, such  that
     $$\|f_j-\phi_j\|_p<\e,$$
for every $j = 1, \ldots, n$, where now each $\mathcal{A}_j$ is a partition of $\Omega$ of measurable sets of positive measure. Let us take $g_1,\ldots,  g_m$ in $L_{p^*}(\mu, X^*)$.
We distinguish two subcases again. If $p>1$, then, by definition of $L_{p^*}(\mu, X^*)$, for each $j\in\{1,\ldots,m\}$, we can find a finite family $\mathcal{C}_j$ of pairwise disjoint measurable sets of positive and finite measure, and $x^*_C\in X^*$ such that the vector-valued simple function  $\eta_k=\sum_{C \in \mathcal{C}_k}x^*_C\chi_C$ satisfies that
$$
\|g_k-\eta_k\|_{q^*}<\e
$$
for every $k = 1, \ldots, m$. Observe that, since $\mu$ is finite, by adding a suitable null characteristic function, we may and do assume that each $\mathcal{C}_k$ is actually a measurable partition of $\Omega$. If $p=1$, then $p^*=\infty$. In that case,    by \cite[p.~97]{Diestel-Uhl-book},  there exists a measurable, bounded and countably valued mapping $\eta_k:\Omega \longrightarrow X^*$ such that
$$
\|g_k-\eta_k\|_\infty<\frac{\e}{2}
$$
for every $k = 1, \ldots, m$. Thus, again there exists a countable partition  $\mathcal{C}_k$ of $\Omega$ and vectors $x^*_C\in X^*$  such that $\eta_k(t)=\sum_{C \in \mathcal{C}_k}x^*_C\chi_C(t)$, for every $t \in \Omega$.

In both cases, we can  find $\mathcal{B}=(B_i)_{i=1}^\infty$ a partition of $\Omega$ of sets of positive measure that is a refinement of all the partitions $\mathcal{A}$, $\mathcal{A}_j$, $\mathcal{C}_k$ for every $j$ and $k$. As in (a), if we define  $P:L_p(\mu, X)\longrightarrow L_p(\mu, X)$ by
$$P(f)=\sum_{i=1}^{\infty}\frac{1}{\mu(B_i)}\Big(\int_{B_i}f d\mu\Big)\, \chi_{B_i},$$
we have that  $P$ is a norm-one projection such that $P(L_p(\mu, X))$ is isometrically isomorphic to $\ell_p(X)$ and
      $$\|P(f_j)-f_j\|_p=\|\phi_j-f_j\|_p<\e,$$ for every $j=1,\ldots, n$.
Furthermore, $L_{p^*}(\mu, X^*)$ is isometrically isomorphic to a subspace of $L_p(\mu, X)^*$ (see e.g.\ \cite[p.~97]{Diestel-Uhl-book}), and
\begin{align*}
\bigl[P^*(x^*\chi_{B_r})\bigr](f)=&\int_{\Omega}x^*\bigl(P(f)(t)\bigr) \chi_{B_r}(t)\,d\mu(t) \\&
=\int_{\Omega}\sum_{i=1}^{\infty}\frac{1}{\mu(B_i)}x^*\left(\int_{B_i}f d\mu\right)\, \chi_{B_i}(t)\chi_{B_r}(t)\,d\mu(t)
\\=& \frac{1}{\mu(B_r)}x^*\left(\int_{B_r}f d\mu\right)\,\int_{\Omega} \chi_{B_r}(t)\,d\mu(t)=\bigl[x^*\chi_{B_r}\bigr](f)
\end{align*}
for every $f$ in  $L_p(\mu, X)$, every $x^*$ in $X^*$  and every $r$. We  know that there exists $x_{ik}^*\in X^*$ such that $\eta_k(t)=\sum_{i=1}^\infty x^*_{ik}\chi_{B_i}(t)$, pointwise in $\Omega$ and convergent with the $\|\cdot\|_{p^*}$-norm for $1<p<\infty$. Hence, for $1<p<\infty$, we obtain that  $P^*(\eta_k)=\eta_k$ for every $k$. For $p=1$ the equality holds too. But we need to do some extra work to prove it. We have
$$
\left[P^*\bigl(\sum_{r=1}^lx_{rk}^*\chi_{B_r}\bigr)\right](f) =\sum_{r=1}^lx_{rk}^*\left(\int_{B_r} f(t)\, d \mu(t)\right)= \int_{\Omega}\sum_{r=1}^lx_r^*(f(t))\chi_{B_r}(t)\, d \mu(t).
$$
But
$$\sum_{r=1}^l|x_{rk}^*(f(t))\chi_{B_r}(t)|\leq \sum_{r=1}^l\|x_{rk}^*\|\|f(t)\|\chi_{B_r}(t)\leq M\|f(t)\|,$$
for every $t$ in $\Omega$ and every $l\in \N$, where $M:=\max\{\|g_1\|_\infty,\ldots,\|g_m\|_\infty\}+\e$.
Thus the series $\sum_{r=1}^\infty x_{rk}^*(f(t))\chi_{B_r}(t)$ is absolutely convergent for every $f$ and every $t$, and we obtain that
$$
\eta_k(t)((f)(t))=\sum_{r=1}^\infty x_{rk}^*(f(t))\chi_{B_r}(t).
$$
Moreover, by the Lebesgue Dominated Theorem
\begin{align*}
\bigl[\eta_k](f)=\int_\Omega \eta_k(f(t))\,d \mu(t)= \int_\Omega  \sum_{r=1}^\infty x_{rk}^*(f(t))\chi_{B_r}(t)\, d \mu(t)
= \sum_{r=1}^\infty \int_\Omega x_{rk}^*(f(t))\chi_{B_r}(t)\, d \mu(t),
\end{align*}
for all $f$ in $L_1(\mu,X)$.
On the other hand, since the series defining $P$ converges in the $\|.\|_1$-norm,
\begin{align*}
\bigl[P^*(\eta_k)\bigr](f)=& \sum_{i=1}^\infty \frac{1}{\mu(B_i)}
\bigl[\eta_k]\left(\left(\int_{B_i}f(u)\, d\mu(u)\right)\chi_{B_i}\right)=
\\=&\sum_{i=1}^\infty \frac{1}{\mu(B_i)}\int_\Omega \sum_{r=1}^\infty  x_{rk}^* \left(\int_{B_i}f(u)\, d\mu(u)\right) \chi_{B_r}(t)\chi_{B_i}(t)\, d \mu(t)
\\=&
\sum_{i=1}^\infty \frac{1}{\mu(B_i)}\int_\Omega x_{ik}^*\left(\int_{B_i}f(u)\, d\mu(u)\right)\chi_{B_i}(t)\, d \mu(t)
\\=&\sum_{i=1}^\infty  x_{ik}^*\left(\int_{B_i}f(u) d\mu(u)\right)=\sum_{i=1}^\infty  \int_\Omega x_i^*( f(u)) \chi_{B_i}(u)\,d\mu(u).
\end{align*}
Thus, for $1\leq p<\infty$,
      $$\|P^*(g_k)-g_k\|_{p^*}\leq \|P^*(g_k)-P^*(\eta_k)\|_{p^*} +\|\eta_k-g_k\|_{p^*}\leq 2\|\eta_k-g_k\|_{p^*}< \e,$$ for every $k=1,\ldots, m$.

To prove (c), we follow the lines of (b). Given  $g_1,\ldots,  g_m$ in $L_\infty(\mu, X)$ and $\e>0$, for each $k\in \{1,\ldots,m\}$ there exists a measurable, bounded  and countably valued mapping $\eta_k:\Omega \longrightarrow X$ such that
$$
\|g_k-\eta_k\|_\infty<\e,
$$
and hence, we have a countable partition  $\mathcal{C}_k$ of $\Omega$ and points  $x_C\in X$ for every $C\in \mathcal{C}_k$ such that $\eta_k(t)=\sum_{C \in \mathcal{C}_k}x_C\chi_C(t)$, for every $t \in \Omega$.
Again, we take an infinite countable partition $\mathcal{B}=\{B_i\,:\, i=\in \N\}$ of $\Omega$ of sets of positive measure that is a refinement of $\mathcal{A}$ and $\mathcal{C}_k$ for all $k$. Finally, if we  define  $P:L_\infty(\mu, X)\longrightarrow L_\infty(\mu, X)$ by
$$P(f)=\sum_{i=1}^{\infty}\frac{1}{\mu(B_i)}\Big(\int_{B_i}f d\mu\Big)\, \chi_{B_i},$$
we have that  $P$ is a norm-one projection such that $P(L_\infty(\mu, X))$ is isometrically isomorphic to $\ell_\infty(X)$ and
$$\|P(g_k)-g_k\|_\infty=\|\eta_k-g_k\|_\infty<\e,$$ for every $k=1,\ldots, m$.
\end{proof}

\begin{proof}[Proof of Theorem~\ref{thr-L1muX}]
If $L_1(\mu)$ is finite-dimensional, the result is a consequence of Lemma~\ref{lemma-characterizationell1X-Y}. So, let us suppose in the rest of the proof that $L_1(\mu)$ is infinite-dimensional.

Let us start with the case when $\mu$ is finite. As $X^*$ has the Radon-Nikod\'{y}m property, we have that $L_\infty(\mu,X^*)= L_1(\mu,X)^*$ (see e.g.\ \cite[Theorem~IV.1.1 in p.~98]{Diestel-Uhl-book}), so Lemma~\ref{lemma-P-P*-when-mu-finite}.(b) provides us with a net $\{P_\lambda\}_{\lambda\in \Lambda}$ of norm-one projections on $L_1(\mu,X)$ such that $\{P_\lambda f\}\longrightarrow f$ in norm for every $f\in L_1(\mu,X)$, $\{P_\lambda^* g\}\longrightarrow g$ in norm for every $g\in L_\infty(\mu,X^*)=L_1(\mu,X)^*$, and $P_\lambda(L_1(\mu,X))$ is isometrically isomorphic to $\ell_1(X)$. Now, we may apply Proposition~\ref{Prop-family-projections}.

If $\mu$ is $\sigma$-finite, we may use \cite[Proposition~1.6.1]{Cembranos} to reduce to the previous case: there is a finite measure $\nu$ such that $L_1(\mu,X)$ is isometrically isomorphic to $L_1(\nu,X)$. Let us also observe that we actually get that there exists a common function $\eta:\R^+\longrightarrow \R^+$, depending only on $X$ and $Y$, such that all the spaces $(L_1(\mu,X),Y)$ have the BPBp for compact operators with the function $\eta$ when $\mu$ is $\sigma$-finite.

Finally, for the general case, we may adapt an argument from the proof of \cite[Proposition~2.1]{ChoiKimLeeMartin}. Let $0<\eps<1$ and $T\in \mathcal{L}(L_1(\mu,X),Y)$ with $\|T\|=1$ and $f_0\in S_{L_1(\mu,X)}$ satisfying
$$
\|T f_0\|> 1- \eta(\eps),
$$
where $\eta$ is the universal function for all $\sigma$-finite measures given in the previous case. Pick a sequence $\{f_n\}_{n\in\N}$ in the unit sphere of $L_1(\mu,X)$ such that $\lim_{n\to \infty}\|T f_n\|=1$. Then, there is a measurable set $A$ such that the measure $\mu|_A$ is $\sigma$-finite and the support of all the $f_n$, $n\geq 0$, are contained in $A$. Then, consider $T_1$ to be the restriction of $T$ to $L_1(\mu|_A,X)$, which satisfies $\|T_1\|=1$ and $\|T_1 f_0\|>1-\eta(\eps)$. By the assumption on $\eta$, there exist a norm-one operator $S_1:L_1(\mu|_A,X)\longrightarrow Y$ and a norm-one vector $g\in L_1(\mu|_A,X)$ such that $\|S_1 g\|=1$, $\|T_1-S_1\|<\eps$ and $\|f_0-g\|<\eps$. Let $P:L_1(\mu,X)\longrightarrow L_1(\mu|_A,X)$ denote the restriction operator. Then $S=S_1P+T(\Id-P)$ is a norm-one operator from $L_1(\mu,X)$ to $Y$, $g$ can be viewed as a norm-one element in $L_1(\mu,X)$ (just extending by $0$), $\|Sg\|=1$, $\|S-T\|<\eps$ and $\|f_0-g\|<\eps$.
\end{proof}

When $X$ is just the base field, we recover the result for pairs of the form $(L_1(\mu),Y)$ from \cite[Corollary~2.4]{ABGKM-JMAA2014}, see Examples~\ref{example-previouslyknown}.(i).

Concrete applications of Theorem~\ref{thr-L1muX} for the vector-valued case can be given using the results of \cite{KimLeeMartin-JMAA2015}.

\begin{cor}
Let $\mu$ be a positive measure and let $X$, $Y$ be Banach spaces. The pair $(L_1(\mu,X),Y)$ has the BPBp for compact operators in the following cases:
\begin{enumerate}
\item[(a)] if $X$ and $Y$ are finite-dimensional;
\item[(b)] if $X^*$ has the Radon-Nikod\'{y}m property, $Y$ is a Hilbert space and the pair $(X,Y)$ has the BPBp for compact operators;
\item[(c)] in particular, if $Y$ is a Hilbert space and $X=c_0$ or $X=L_p(\nu)$ for any positive measure $\nu$ and $1< p < \infty$.
\end{enumerate}
\end{cor}

\begin{proof} (a). When $X$ and $Y$ are finite-dimensional, it is shown in \cite[Proposition~7]{KimLeeMartin-JMAA2015} that the pair $(\ell_1(X),Y)$ has the BPBp. By finite-dimensionality, Lemma~\ref{lemma-characterizationell1X-Y} gives then that the pair $(\ell_1(X),Y)$ has the BPBp for compact operators. Now, Theorem~\ref{thr-L1muX} applies as $X$ is Asplund.

(b). If we only consider finite convex sums instead of convex series, we may repeat the proof of \cite[Proposition~9]{KimLeeMartin-JMAA2015} but using only compact operators to get item (i) of Lemma~\ref{lemma-characterizationell1X-Y}. Then, we have that $(\ell_1(X),Y)$ has the BPBp for compact operators. If $X^*$ has the Radon-Nikod\'{y}m property, Theorem~\ref{thr-L1muX} finishes the proof.

(c) follows from (b) and Examples~\ref{example-previouslyknown}.
\end{proof}

The proof of Theorem~\ref{thr-L1muX} can be easily adapted to pairs of the form $(L_p(\mu,X),Y)$ for $1<p<\infty$, but only when the measure $\mu$ satisfies that $L_1(\mu)$ is infinite-dimensional.

\begin{prop}\label{prop-LpmuX} Let $1<p<\infty$, let $\mu$ be a positive measure such that $L_1(\mu)$ is infinite-dimensional, let $X$ be a Banach space such that $X^*$ has the Radon-Nikod\'{y}m property, and let $Y$ be a Banach space. If the pair $(\ell_p(X),Y)$ has the BPBp for compact operators, then so does the pair $(L_p(\mu,X),Y)$.
\end{prop}

Let us observe that, in this case, the scalar-valued version of the result has no interest, as the spaces $L_p(\mu)$ are uniformly convex for $1<p<\infty$ and we may use Examples~\ref{example-previouslyknown}.(b).

The last applications deal with modifying the range space.

\begin{theorem}\label{theorem:x-lp-mu-Y}
Let $X$, $Y$ be Banach spaces.
\begin{enumerate}
\item[(a)] For $1\leq p <\infty$, if  the pair $(X,\ell_p(Y))$ has the BPBp for compact operators, then so does $(X,L_p(\mu,Y))$ for every positive measure $\mu$ such that $L_1(\mu)$ is infinite-dimensional.
\item[(b)] If the pair $(X,Y)$ has the BPBp for compact operators, then so does $(X,L_\infty(\mu,Y))$ for every $\sigma$-finite positive measure $\mu$.
\item[(c)] If the pair $(X,Y)$ has the BPBp for compact operators, then so does $(X,C(K,Y))$ for every compact Hausdorff topological space $K$.
\end{enumerate}
\end{theorem}

We need the following result which reminds Lemma \ref{lemma-c0Y}, but for range spaces.

\begin{lemma}\label{lemma-ell-p-co-range}
Let $X$, $Y$ be Banach spaces and let $\eta:\R^+\longrightarrow \R^+$ be a function. The following are equivalent:
    \begin{enumerate}
    \item[(i)] the pair $(X,Y)$ has the BPBp for compact operators with the function $\eta$,
    \item[(ii)] the pairs $(X,\ell_\infty^m(Y))$ with $m\in \N$ have the BPBp for compact operators with the function $\eta$,
    \item[(iii)] the pair $(X,c_0(Y))$ has the BPBp for compact operators with the function $\eta$,
    \item[(iv)] the pair $(X,\ell_\infty(Y))$ has the BPBp for compact operators with the function $\eta$.
        \end{enumerate}
\end{lemma}

\begin{proof}
(i) implies (ii), (i) implies (iii), and (i) implies (iv) can be proved adapting the proof of \cite[Proposition~2.4]{ACKLM} to compact operators. Finally, the fact that any of the assertions (ii), (iii) or (iv) implies (i) is a consequence of Lemma~\ref{lemma-X-Y_1-p-Y_2}.(b).
\end{proof}

We are now ready to present the proof of the theorem.

\begin{proof}[Proof of Theorem~\ref{theorem:x-lp-mu-Y}]
(a). Fix $1\leq p< \infty$. If $L_1(\mu)$ is infinite-dimensional, Lemma~\ref{lemma-P-P*-when-mu-finite}.(a) provides a net $\{Q_\lambda\}_{\lambda \in \Lambda}$ of norm-one projections on $L_p(\mu,X)$ such that $\{Q_\lambda f\} \longrightarrow f$ in norm for every $f \in L_p(\mu,Y)$ and $Q_\lambda(L_p(\mu,X))$ is isometrically isomorphic to $\ell_p(Y)$. Now, the result follows from Proposition~\ref{prop-range-space}.

(b). If $L_\infty(\mu)$ is finite-dimensional, the result is a consequence of Lemma~\ref{lemma-ell-p-co-range}. Otherwise, if $L_\infty(\mu)$ is infinite-dimensional, we may and do suppose that the measure is finite by using \cite[Proposition~1.6.1]{Cembranos}.
Lemma~\ref{lemma-P-P*-when-mu-finite}.(c) provides a net $\{Q_\lambda\}_{\lambda \in \Lambda}$ of norm-one projections on $L_\infty(\mu,X)$ such that $\{Q_\lambda f\} \longrightarrow f$ in norm for every $f \in L_\infty(\mu,Y)$ and $Q_\lambda(L_\infty(\mu,X))$ is isometrically isomorphic to $\ell_\infty(Y)$. Now, Lemma~\ref{lemma-ell-p-co-range} gives that all the pairs $(X,Q_\lambda(L_\infty(\mu,X)))$ have the BPBp for compact operators with the same function, and so the result follows from Proposition~\ref{prop-range-space}.

(c). Following step-by-step the proof of \cite[Theorem~4]{WBJohnson}, by using peak partitions of the unit and extending the scalar-valued case to the vector-valued case, we may find a net $\{Q_\lambda\}_{\lambda\in \Lambda}$ of norm-one projections on $C(K,Y)$ such that $\{Q_\lambda f\} \longrightarrow f$ in norm for every $f \in C(K,Y)$ and $Q_\lambda(C(K,Y))$ is isometrically isomorphic to $\ell_\infty^m(Y)$. Now, the result follows from Lemma~\ref{lemma-ell-p-co-range} and Proposition~\ref{prop-range-space}.
\end{proof}

Some consequences of Theorem~\ref{theorem:x-lp-mu-Y} are the following.

\begin{cor}
Let $X$, $Y$ be Banach spaces, let $K$ be a compact Hausdorff topological space, let $\mu$ be a positive measure and let $\nu$ be a $\sigma$-finite positive measure.
\begin{enumerate}
\item[(a)] If $Y$ has property $\beta$, then $(X,L_\infty(\mu,Y))$ and $(X,C(K,Y))$ have the BPBp for compact operators.
\item[(b)] If $Y$ has the AHSP, then so do $L_\infty(\nu,Y)$ and $C(K,Y)$.
\item[(c)] For $1\leq p<\infty$, if $\ell_p(Y)$ has the AHSP and $L_1(\mu)$ is infinite-dimensional, then so does $L_p(\mu,Y)$.
\end{enumerate}
\end{cor}

\begin{proof}
(a) is a direct consequence of Example~\ref{example-previouslyknown}.(a) and Theorem~\ref{theorem:x-lp-mu-Y}. For (b) and (c), use Corollary~\ref{cor:characAHSP} that a Banach space $Z$ has the AHSP if and only if the pair $(\ell_1,Z)$ has the BPBp for compact operators. With this in mind, both affirmations are direct consequence of Theorem~\ref{theorem:x-lp-mu-Y}.
\end{proof}

\vspace*{0.2cm}

\noindent \textbf{Acknowledgment:\ }

The authors would like to thank Bill Johnson for kindly answering several inquiries.

%\vspace*{1cm}

\end{document}